\documentclass{amsart}
\usepackage{amssymb,color}

\usepackage{graphicx,subfigure}

\title{Schur times Schubert via the Fomin-Kirillov algebra}

\author{Karola M\'esz\'aros, Greta Panova, Alexander Postnikov}

\address{Karola M\'esz\'aros, Department of Mathematics, 
Cornell Univ., Ithaca, NY, 14853}
\email{karola@math.cornell.edu}

\address{Greta Panova, Department of Mathematics, UCLA, Los Angeles, CA, 90095}
\email{panova@math.ucla.edu}

\address{Alexander Postnikov, Department of  Mathematics, 
MIT, Cambridge, MA, 02139}
\email{apost@math.mit.edu}

\date{\today}



\thanks{M\'esz\'aros is partially supported by an NSF Postdoctoral Research Fellowship DMS 1103933;
Panova is partially supported by a Simons Postdoctoral Fellowship; Postnikov is 
partially supported by NSF grant DMS-6923772}

\newcommand\noshow[1]{}
\newcommand{\old}[1]{}

\newtheorem{theorem}{Theorem}
\newtheorem*{corollary*}{Corollary}
\newtheorem*{theorem*}{Theorem}

\newtheorem{proposition}[theorem]{Proposition}
\newtheorem{lemma}[theorem]{Lemma}
\newtheorem{corollary}[theorem]{Corollary}
\newtheorem{conjecture}[theorem]{Conjecture}

\theoremstyle{remark}

\theoremstyle{definition}

\def\b{{\bar{\lambda}}}
\def\f_H{{\bf w}}

\def\E{\mathcal{E}}

\def\Z{\mathbb{Z}}

\def\r{{\rm row}}
\def\c{{\rm col}}
\def\L{\mathcal{L}}

\def\D{\mathcal{D}}
\def\Dk{\D_{k\times(n-k)} }

\def\l{{\lambda}}

\def\be{\begin{equation}}
\def\ee{\end{equation} }

\def\bl{\begin{lemma}}
\def\el{\end{lemma} }

\def\t{\theta}
\def\la{\lambda}
\def\x{{\bf x}}

\def\C{\mathbb{C}}
\def\H{{\rm H}}

\def\S{\mathfrak{S}}

\def\Z{\mathbb{Z}}
\def\QH{{\rm QH}}

\def\s{\sigma}

\def\Fl{\mathit{Fl}}

\begin{document}

\begin{abstract}
We study multiplication of any Schubert polynomial $\S_w$ by a 
Schur polynomial $s_{\lambda}$ (the Schubert polynomial of a Grassmannian permutation) and the expansion of this product in the ring of Schubert polynomials.  We derive explicit nonnegative combinatorial expressions for the
expansion coefficients  for
certain special partitions $\lambda$, including hooks and the $2\times 2$ box. We also prove combinatorially the existence of such nonnegative expansion when the Young diagram of $\la$ is a hook plus a box at the $(2,2)$ corner.
We achieve this by evaluating Schubert polynomials at the Dunkl elements of the Fomin-Kirillov algebra and proving special cases of the nonnegativity
conjecture of Fomin and Kirillov.

This approach works in the more general setup of the (small) quantum cohomology
ring of the complex flag manifold and the corresponding 
(3-point) Gromov-Witten invariants.
We provide an algebro-combinatorial proof of the nonnegativity of
the Gromov-Witten invariants in these cases, and present combinatorial
expressions for these coefficients.  
\end{abstract}

\maketitle

\section{Brief Introduction}

An outstanding open problem of modern Schubert Calculus is to 
find a combinatorial rule for the expansion coefficients 
$c_{uv}^w$ of the products of
Schubert polynomials (the generalized Littlewood-Richardson coefficients),  
and thus provide an algebro-combinatorial proof of their positivity.
The coefficients $c_{uv}^w$ are the intersection numbers of the Schubert 
varieties in the complex flag manifold $\Fl_n$.
They play a role in algebraic geometry, representation
theory, and other areas.

We establish combinatorial rules for the coefficients $c_{uv}^w$ when $u$ are certain special permutations. This  confirms the insight of  Fomin and Kirillov \cite{FK}, who introduced a certain noncommutative quadratic
algebra $\E_n$  in the  hopes of finding a combinatorial rule for the
generalized Littlewood-Richardson coefficients $c_{uv}^w$.  A combinatorial proof of the  nonnegativity conjecture of Fomin and Kirillov   \cite[Conjecture 8.1]{FK} would directly yield a  combinatorial rule for the $c_{uv}^w$'s. We prove several special cases of this important conjecture, thereby  obtaining the desired rule for a set of the $c_{uv}^w$'s.

One benefit of the approach via the Fomin-Kirillov algebra is that it can be
easily extended and adapted to the (small) quantum cohomology ring of the
flag manifold $\Fl_n$ and the corresponding (3-point) Gromov-Witten invariants.  These
Gromov-Witten invariants extend the generalized Littlewood-Richardson
coefficients.   They count the numbers of rational curves of a given degree that
pass through given Schubert varieties, and play a role in enumerative algebraic geometry.

Some progress on the nonnegativity conjecture  \cite[Conjecture 8.1]{FK}  was made in [P], where the Fomin-Kirillov
algebra was applied for giving a Pieri formula for the quantum cohomology
ring of $\Fl_n$.  However the problem of finding a combinatorial rule
for the generalized Littlewood-Richardson coefficients and the Gromov-Witten
invariants of $\Fl_n$ via the Fomin-Kirillov algebra (or by any other means)
still remains widely open in the general case.

The main result of  this paper is the proof of  several special cases the of nonnegativity conjecture  of Fomin and Kirillov \cite[Conjecture 8.1]{FK}.  It is worth noting that before our present results, the only progress on the nonnegativity conjecture of Fomin and Kirillov were those given in [P], over a decade ago.  Until now, other means for computing these coefficients have lead only to one of our special cases, see \cite{sottile}. Other cases when two of the permutations are restricted have been studied by Kogan in \cite{kog}. Our current paper is a significant generalization of the results given in [P]. While our theorems still only address  special cases of the nonnegativity conjecture, the results we present  are new and are a compelling  step forward.
\medskip

The outline of this paper is as follows. In Section \ref{sec:back} we explain more of the background as well as state the nonnegativity conjecture of Fomin and Kirillov \cite{FK} and a simplified version of our results.  In Section \ref{sec:hooks} we give an expansion of the
product of any  Schubert polynomial with a Schur function
indexed by a hook in terms of Schubert polynomials by proving the corresponding case of the nonnegativity conjecture. In Section
\ref{sec:hom} we explain what the previous implies about the multiplication of
certain Schubert classes in the quantum and $p$--quantum cohomology rings. Finally, Section
\ref{sec:other} is devoted to proving the nonnegativity of the structure
constants for quantum Schubert polynomials in the case of Schur function
$s_\lambda$
indexed by a hook plus a box, that is $\lambda=(b,2,1^{a-1})$,  
and deriving explicit
expansions of $s_{\lambda}(\theta_1,\ldots,\theta_k)$ 
when $\lambda=(2,2), r^k, (n-k)^r$.

\section{Background and definitions}
\label{sec:back}

 We start with a brief discussion of the cohomology ring 
of the flag manifold, 
the Schubert polynomials, the Fomin-Kirillov algebra $\E_n$, and the 
Fomin-Kirillov nonnegativity conjecture in the classical (non-quantum) case;
see \cite{BGG, fulton, mac, manivel, FK} for  more details.
Then we discuss the quantum extension, see \cite{FGP, P} for more details. We also explain how our results fit in this general scheme.

\subsection{The Fomin-Kirillov nonnegativity conjecture}

According to the classical result by Ehresmann~\cite{Ehr},
the cohomology ring $\H^*(\Fl_n) = \H^*(\Fl_n, \C)$ of the flag 
manifold $\Fl_n$ has the linear basis of Schubert classes $\sigma_w$
labeled by permutations $w\in S_n$ of size $n$.
On the other hand, Borel's theorem \cite{borel} says that the cohomology 
ring $\H^*(\Fl_n)$ is isomorphic to the quotient
of the polynomial ring
$$
\H^*(\Fl_n) \simeq \C[x_1,\dots,x_n]/ \left<e_1,\dots,e_n\right>,
$$
where $e_i = e_i(x_1,\dots,x_n)$
are the elementary symmetric polynomials.

Bernstein, Gelfand, and Gelfand~\cite{BGG} and Demazure~\cite{Dem} related
these two descriptions of the cohomology ring of $\Fl_n$.  Lascoux and
Sch\"utzenberger~\cite{ls} then constructed the Schubert polynomials
$\S_w\in\C[x_1,\dots,x_n]$, $w\in S_n$, whose cosets modulo the 
ideal $\left<e_1,\dots,e_n\right>$ correspond to
the Schubert classes $\s_w$ under Borel's isomorphism.

The generalized Littlewood-Richardson coefficients $c_{uv}^w$
are the expansion coefficients of 
products of the Schubert classes in the cohomology ring $\H^*(\Fl_n)$:
$$
\sigma_u \, \sigma_v  = \sum_{w\in S_n} c_{uv}^w \, \sigma_w.
$$
Equivalently, they are the expansion coefficients 
of products of the Schubert polynomials:
$\S_u \, \S_v  = \sum_w c_{uv}^w\, \S_w$.

The {\it Fomin-Kirillov algebra\/} $\E_n$, introduced in \cite{FK},
is the associative algebra over $\C$
generated by $x_{ij}$, $1\leq i<j\leq n$, with 
the following relations:
$$
\begin{array}{l}
\displaystyle
x_{ij}^2 = 0,\\[.05in]
\displaystyle
x_{ij} \, x_{jk} = x_{ik}\,  x_{ij} + x_{jk} \, x_{ik},
\qquad
\displaystyle
x_{jk} \, x_{ij} = x_{ij} \, x_{ik}  + x_{ik} \, x_{jk},\\[.05in]
x_{ij} \, x_{kl} = x_{kl} \, x_{ij}\qquad
\textrm{for distinct } i,j,k,l.
\end{array}
$$
It comes equipped with the {\it Dunkl elements}
$$
\theta_i = - \sum_{j<i} x_{ji}  + \sum_{k>i} x_{ik}.
$$
It is not hard to see from the relations in $\E_n$ that the Dunkl
elements commute pairwise $\theta_i \theta_j = \theta_j \theta_i$ (\cite[Lemma 5.1]{FK}).

The Fomin-Kirillov algebra $\E_n$ acts on the cohomology ring $\H^*(\Fl_n)$
by the following {\it Bruhat operators:}
$$
x_{ij}:\sigma_w\longmapsto \left\{
\begin{array}{cl}
\sigma_{w \, s_{ij}}, &\textrm{if } \ell(w\,s_{ij}) = \ell(w)+1\\
0&\textrm{otherwise,}
\end{array}
\right.
$$
where $s_{ij}\in S_n$ denotes the transposition of $i$ and $j$,
and $\ell(w)$ denotes the length of a permutation $w\in S_n$.

The classical Monk's formula says that the Dunkl elements $\theta_i$ 
act on the cohomology ring $\H^*(\Fl_n)$
as the operators of multiplication by the $x_i$ (under Borel's isomorphism),
$\theta_i:\sigma_w\mapsto x_i \,\sigma_w$.
The commutative subalgebra of $\E_n$ generated by the Dunkl elements 
$\theta_i$ is canonically isomorphic to the cohomology ring $\H^*(\Fl_n)$.

Since the Dunkl elements $\theta_i$ commute pairwise, one can evaluate
a Schubert polynomial (or any other polynomial) at these elements
$\S_w(\theta_1,\dots,\theta_n)\in \E_n$.

It follows immediately from the definitions that these evaluations
act on the cohomology ring of $\Fl_n$ as
$$
\S_u(\theta_1,\dots,\theta_n) 
: \sigma_v \mapsto \sum_{w\in S_n} c_{uv}^w \, \sigma_w.
$$
Indeed, $\S_u(\theta_1,\dots,\theta_n)$ acts on the cohomology 
ring $\H^*(\Fl_n)$ as the operator of multiplication by the Schubert class
$\sigma_u$.

This implies that if there exists an explicit expression of the evaluation
$\S_u(\theta_1,\dots,\theta_n)$ in which every monomial in the generators $x_{ij}$ ($i<j$) has a nonnegative coefficient, such expression immediately gives a combinatorial rule
for the generalized Littlewood-Richardson coefficients $c_{uv}^w$
for all permutations $v$ and $w$.

Let $\E_n^+ \subset \E_n$ be the cone of all nonnegative linear combinations
of monomials in the generators $x_{ij}$, $i<j$,  of $\E_n$.
Fomin and Kirillov formulated the following Nonnegativity Conjecture in light of the search for a combinatorial proof of the positivity of $c^w_{uv}$. 

\begin{conjecture}
\label{c1} 
{\rm \cite[Conjecture~8.1]{FK}} \
For any permutation $u\in S_n$, the evaluation
$\S_u(\theta_1,\dots,\theta_n)$ belongs to the nonnegative cone $\E_n^+$.
\end{conjecture}

\subsection{New results}
Our main result, in a simplified form, is a proof of some special cases of Conjecture ~\ref{c1} beyond the Pieri rule proven in \cite{P}:

\begin{theorem}  For  a Grassmannian permutation $u\in S_n$, whose code $\lambda(u)$ is a hook shape or a hook shape with a box added in position $(2,2)$,  the evaluation
$\S_u(\theta_1,\dots,\theta_n)$ belongs to the nonnegative cone $\E_n^+$. 
\end{theorem}

Moreover, we give \emph{an explicit combinatorial expansion}  in \textbf{Theorems \ref{thm:hooks}} and \textbf{\ref{thm-p} }when $\la=(s,1^{t-1})$ is a hook, when $\la=(2,2)$ (\textbf{Theorem \ref{thm:twobytwo}})  and $\la=(n-k)^r$ or $\la=t^k$ (\textbf{Proposition \ref{prop:rectangles}}). We also prove the existence of a nonnegative expansion when $\la=(b,2,1^{a-1})$ is a hook plus a box at $(2,2)$ in \textbf{Theorem \ref{tim:hook+box}}. 

{\bf Remark.} These results provide {\it combinatorial\/} proofs of the
nonnegativity of the expansion coefficients $c_{uw}^v$ of the product
$\S_u s_w$ in terms of $\S_v$ and, moreover, explicit combinatorial
rules for the coefficients $c_{uw}^v$ for special permutations $u$ as
above and arbitrary permutations $v, w$.

Our main tools come from the following connection with symmetric functions.

Schubert polynomials for \emph{Grassmannian permutations} are actually Schur functions, see e.g. \cite{manivel} and \cite{mac}. Grassmannian permutations, by definition, are permutations $w$ with a unique descent. There is a straightforward bijection between such permutations and partitions which fit in the $k\times (n-k)$ rectangle, where $k$ is the position of the descent. Given a permutation $w$ with a unique descent at position $k$ we define the corresponding partition $\la(w)$, \emph{the code of $w$}, as follows
$$\lambda(w)_i = w_{k+1-i} - (k+1-i).$$ 
In the other direction, given $k$ and $\la$ of at most $k$ parts with $\la_1 \leq n-k$ we define a permutation $w(\la,k)$ by 
\begin{equation}\label{d:perm_la}
w(\la,k)_i =  \la_{k+1-i}+i \quad \text{for $i=1,\ldots,k$, and }\; w_{k+1}\ldots w_n = [n]\setminus \{w_1,\ldots,w_k\},
\end{equation}
where the last $n-k$ elements of $w(\la,k)$ are arranged in increasing order. Clearly these operations are inverses of each other. 
It is well-known that 
\emph{if $w$ is a Grassmannian permutation with descent at $k$, then
$$\S_w(x_1,\ldots,x_n) = s_{\la(w,k)}(x_1,\ldots,x_k).$$ }



In \cite{P}, the problem of evaluating $\S_u$ at the Dunkl elements was solved in the case when $\S_u$ is the elementary
and the complete homogenous symmetric polynomials $e_i(x_1,\dots,x_k)$ 
and $h_i(x_1,\dots,x_k)$ in $k<n$ variables, i.e. when the Young diagram of $\la$ is a row or column. We cite this below as Theorem \ref{th:pieri}.



\subsection{Quantum cohomology}
The story generalizes to the (small) quantum cohomology ring $\QH^*(\Fl_n)
=\QH^*(\Fl_n,\C)$ of
the flag manifold $\Fl_n$ and the corresponding \emph{3-point Gromov-Witten invariants}.  
As a vector space, the quantum cohomology is isomporphic to 
\[
  \QH^*(\Fl_n)\cong \H^*(\Fl_n)\otimes \C[q_1,\dots,q_{n-1}].
\]
Thus the Schubert classes $\sigma_w$, $w\in S_n$, form a 
linear basis of $\QH^*(\Fl_n)$ over $\C[q_1,\dots,q_{n-1}]$.
However, the multiplicative structure in~$\QH^*(\Fl_n)$ 
is quite different from that of the usual cohomology. 

A quantum analogue of Borel's theorem was suggested by Givental and
Kim~\cite{giv-kim}, and then justified by Kim~\cite{Kim} and
Ciocan-Fontanine~\cite{ciocan}.  
They showed that the quantum cohomology
ring $\QH^*(\Fl_n)$ is canonically isomorphic to the quotient
\begin{equation}
  \label{eq:q-factor}
\QH^*(\Fl_n)\simeq
  \C[x_1,\dots,x_n;q_1,\dots,q_{n-1}]\,/\left<E_1,E_2,\dots,E_n\right>,
\end{equation}
where  the $E_i\in\C[x_1,\dots,x_n;q_1,\dots,q_{n-1}]$ are 
are certain $q$-deformations of the elementary symmetric
polynomials $e_i=e_i(x_1,\dots,x_n)$, and they specialize to the 
$e_i$ when $q_1=\dots=q_{n-1}=0$.


Analogs of the Schubert polynomials for the quantum cohomology, called 
the {\it quantum Schubert polynomials\/} $\S_w^q$,
were constructed in \cite{FGP}.
%
%
According to \cite{FGP}, the cosets of these polynomials 
$\S_w^q$ represent the Schubert classes $\sigma_w$
in $\QH^*(\Fl_n)$  under the isomorphism (\ref{eq:q-factor}).
This provides an extension of results of Bernstein-Gelfand-Gelfand
\cite{BGG} to the quantum cohomology,
and reduces the geometric problem of multiplying the Schubert classes 
in the quantum cohomology and calculating the 3-point Gromov-Witten invariants
to the combinatorial problem of expanding products of the quantum Schubert
polynomials.

A quantum deformation of the algebra $\E_n$, denoted by  $\E_n^q$, was also
constructed in \cite{FK}, as well as the more general $\E_n^p$. 
Briefly, $\E_n^p$ is defined similarly to $\E_n$: it is generated by $x_{ij}$ and $p_{ij}$ with the additional (modified) relations that 
$$x_{ij}^2=p_{ij} \;, \text{ and }\;  [p_{ij},p_{kl}] = [p_{ij},x_{kl}]=0\,,\quad\textrm{for any }
  i,j,k,\textrm{ and }l\,,$$ where [,] is the commutator.
Then  $\E_n$ is the quotient of the algebra $\E_n^p$ 
modulo the ideal generated by the $p_{ij}$.  
Also let $\E_n^q$ be the the quotient of $\E_n^p$ modulo the 
ideal generated by the $p_{ij}$ with
$|i-j|\geq 2$.  The image of $p_{i\,i+1}$ in $\E_n^q$ is denoted~$q_i$. 

  These algebras also come with pairwise commuting
Dunkl elements $\theta_i$ (defined as in $\E_n$).  The generators of the algebra $\E_n^q$ act on the 
quantum cohomology ring $\QH^*(\Fl_n)$ by simple and explicit 
quantum Bruhat operators.
It was shown in \cite{P} that the commutative subalgebra of $\E_n^q$
generated by the Dunkl elements $\theta_i$ is canonically isomorphic to the
quantum cohomology ring of $\Fl_n$.
Similar to the above discussion for the classical case, 
a way to express the evaluation
of a quantum Schubert polynomial $\S_u^q(\theta_1,\dots,\theta_n)
\in \E_n^q$ as a nonnegative expression in the generatiors of $\E_n^q$
immediately implies a \emph{combinatorial rule for the 3-point Gromov-Witten invariants};
see \cite{P} for more details.

The \emph{$p$--quantum elementary symmetric polynomials} $E_k(x_{i_1}, ...,
x_{i_m}; p)$ are defined in \cite{P}.  (Here $\{i_1,\dots,i_m\}$ is a
subset of $[n]$.)   These polynomials specialize to the usual
elementary symmetric polynomials $e_k(x_{i_1},..., x_{i_m})$ when all
$p_{ij}=0$.

The following Pieri rule will be instrumental for the proofs in the current paper.
\begin{theorem} \cite[Theorem 3.1]{P} {\rm (Quantum Pieri's formula)} \
\label{th:pieri}
Let $I$ be a subset in $\{1,2,\dots,n\}$, and
let $J=\{1,2,\dots,n\}\setminus I$.  Then, for $k\geq 1$, 
the evaluation $E_k(\theta_I;p)\in \E_n^p$ of the $p$-quantum elementary 
symmetric polynomial at the Dunkl elements $\theta_i$ is given by 
\begin{equation}
  \label{eq:pieri}
  E_k(\theta_I;p) = \sum x_{a_1\,b_1}x_{a_2\,b_2}\cdots x_{a_k b_k},
\end{equation}
where the sum is over all sequences of integers $a_1,\dots,a_k,b_1,\dots,b_k$ 
such that {\rm ({i})} $a_j\in I,$ $b_j\in J$, for $j=1,\dots,k$; 
{\rm ({i}{i})} the $a_1,\dots,a_k$ are distinct; 
{\rm ({i}{i}{i})}  $b_1\leq \cdots\leq b_k$.
\end{theorem}

Specializing $p_{ij}=0$, one obtains $E_k(x_I;0)=e_k(x_I)$, the 
usual elementary symmetric polynomial.

A completely analogous statement holds for the homogeneous symmetric functions $h_k$, whose $p-$quantum definition is as the corresponding $p-$quantum Schubert polynomial. The expansion of ($p-$quantum) $h_k(\t_I)$ is obtained by interchanging the roles of the first and second indices in the variables $x_{ij}$ in \eqref{eq:pieri}, i.e. 
\begin{equation}
 h_k(\theta_I) = \sum x_{a_1\,b_1}x_{a_2\,b_2}\cdots x_{a_k b_k},
\end{equation}
where the sum is over all sequences of integers $a_1,\dots,a_k,b_1,\dots,b_k$ 
such that {\rm ({i})} $a_j\in I,$ $b_j\in J$, for $j=1,\dots,k$; 
{\rm ({i}{i})} the $b_1,\dots,b_k$ are distinct; 
{\rm ({i}{i}{i})}  $a_1\leq \cdots\leq a_k$.

\medskip

Following the definition of quantum Schubert
polynomials  $\frak{S}_w^q$ in \cite{FGP}, we define the more 
general \emph{$p$-quantum Schubert  polynomials  $\frak{S}_w^p$ }, as follows. 
Let 
$$
e_{i_1, \ldots, i_{n-1}}=
e_{i_1}(x_1)e_{i_2}(x_1, x_2)\cdots e_{i_{n-1}}(x_1, \ldots,
x_{n-1}),
$$
where $i_j \in \{0, 1, 2, \ldots, j\}$, for $j \in [n-1]$, and
$e_0^k=1$. Similarly, let 
$$
E^p_{i_1, \ldots, i_{n-1}}=E_{i_1}^1E_{i_2}^2\cdots
E_{i_{n-1}}^{n-1}=E_{i_1}(x_1; p)E_{i_2}(x_1, x_2; p)\cdots E_{i_{n-1}}(x_1,
\ldots, x_{n-1}; p). 
$$

One can uniquely write a Schubert polynomial $\S_w$ as a linear 
combination of the $e_{i_1,\dots,i_{n-1}}$:
\be \label{schub_e}
\S_w=\sum \alpha_{i_1, \ldots, i_{n-1}}\, e_{i_1, \ldots, i_{n-1}}.
\ee
The {\it $p$-quantum Schubert polynomial} $\S_w^p$ is then defined as  
\be \label{schub_e_q}
\S_w^p=\sum \alpha_{i_1, \ldots, i_{n-1}}\, E^p_{i_1, \ldots, i_{n-1}}.
\ee

For any $\la$ we define the {\it $p$-quantum Schur polynomial\/} as
$$
s_\lambda^p(x_1,\dots,x_k) = \S_{w(\la,k)}^p.
$$

Note that the $p$-quantum Schubert polynomial $\S_w^p$ 
specializes to the quantum Schubert polynomial $\S_w^q$ from \cite{FGP} 
if we set 
$p_{i\,i+1}=q_i$, $i=1,2,\dots,n-1$, and $p_{ij}=0$, for $|i-j|\geq 2$.

We can now give the quantum Nonnegativity Conjecture of Fomin and Kirillov.

\begin{conjecture} 
\cite[Conjecture 14.1]{FK} 
\label{c2} For any  $w \in {S}_n$, the evaluation
of the quantum Schubert polynomial 
$\S_w^q(x_1,\dots,x_n;q_1,\dots,q_{n-1})$ at the Dunkl elements $\theta_i$
$$
\frak{S}_w^q(\t)=\frak{S}^q_w(\t_1, \ldots, \t_{n}; q_1,\dots,q_{n-1}) 
\in \E_n^q
$$
can be written as a nonnegative linear combination of monomials 
in the generators $x_{ij}$, for $i<j$, of the Fomin-Kirillov
algebra $\E_n^q$.
\end{conjecture}

In this paper we prove the quantum and $p$--quantum analogues of all our results and show that the expansions in $\E^p_n$ and $\E_n$ coincide.

\begin{theorem}  For  $w\in S_n$, for which $\lambda(w)$, the code of $w$, is a hook shape or a hook shape with a box added in position $(2,2)$,   the evaluation
of the quantum Schubert polynomial 
$\S_w^q(x_1,\dots,x_n;q_1,\dots,q_{n-1})$ at the Dunkl elements $\theta_i$
$$
\frak{S}_w^q(\t)=\frak{S}^q_w(\t_1, \ldots, \t_{n}; q_1,\dots,q_{n-1}) 
\in \E_n^q$$
can be written as a nonnegative linear combination of monomials 
in the generators $x_{ij}$, for $i<j$, of the Fomin-Kirillov
algebra $\E_n^q$.
\end{theorem}

\medskip



\section{The nonnegativity conjecture for $s_{\l}$ where $\l$ is a hook}
\label{sec:hooks}

This section concerns the Nonnegativity Conjecture for
$\S_w = s_{\lambda}(x_1,\dots,x_k)$, where $\lambda$ is a hook shape.
Note that an extension of Pieri's formula to hook shapes was given 
by Sottile \cite[Theorem 8, Corollary 9]{sottile} using a different approach.

We prove Conjectures~\ref{c1} and~\ref{c2} 
for Grassmannian permutations $w(\la,k)$ (see \eqref{d:perm_la}), where $\la =(s,1^{t-1})$ is a hook,  
 by giving an explicit expansion for $\frak{S}_w(\t)$ which is in $\E_n^+$ and then using Lemma \ref{obv} to show that this same expansion also equals $\frak{S}^p_w(\t)$.

Consider a rectangle $R_{k\times(n-k)}$ whose rows are indexed by $\{1, \ldots, k\}$ and whose columns are indexed by $\{k+1, \ldots, n\}$. A \textit{box} of this rectangle is specified by its row and column index. A \textit{diagram} $D$ in this rectangle is a collection of boxes. Denote by  $\r(D)$ and $\c(D)$ the number of rows and number of columns which contain a box of $D$, respectively.
We say that a diagram $D$ is a \textit{forest}, if the graph, which we obtain by considering $D$'s boxes as the vertices and connecting two vertices if the corresponding boxes are in the same row or same column and there is no box directly between them, is a forest. See Figure \ref{forest} for an example.

\begin{figure}
\begin{center}
\includegraphics[scale=1.2]{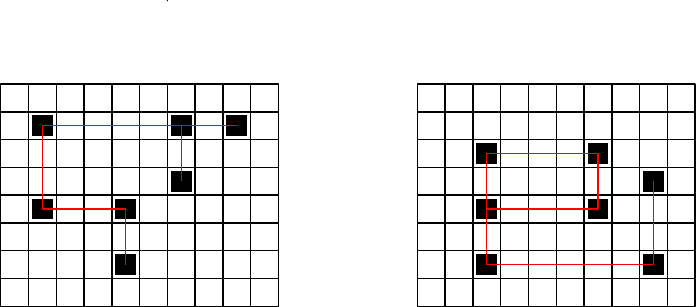}
 \caption{Examples of diagrams. The black boxes indicate the diagrams in the two $8\times 10$ rectangles.  The red edges are the edges of the graph whose vertices are the black boxes and where boxes are connected by an edge if they are in the same row or same column and there is no box directly between them. Thus, the left hand side diagram is a forest, whereas the right hand side diagram is not.}
 \label{forest}
 \end{center}
\end{figure}

Denote by  $ \D_{k\times(n-k)}$ the set of forests which fit into $R_{k\times(n-k)}$. A labeling of a diagram $D \in \Dk$ is an assignment of the numbers $1, 2, \ldots, |D|$ to its boxes (one number to each box). Obviously, there  are $|D|!$ distinct labelings of $D$. Let $D_L$ denote a labeling of $D$. Define  the monomial $x^{D_L}$ in the natural way: if the number $k$ is assigned to the box in row $i_k$ and column $j_k$ in the labeling $D_L$, then $x^{D_L}:=x_{i_1j_1}\cdots x_{i_{|D|}j_{|D|}}$.
 If  for two labelings $D_L\neq D_{L'}$ of $D$ we have that  $x^{D_L}=x^{D_{L'}}$ in $\E_n$,  and in order to get the equality $x^{D_L}=x^{D_{L'}}$ only commutation relations (5) were used, we consider the labelings $D_{L}$ and  $D_{L'}$ equivalent and write $D_L \sim_D D_{L'}$.  The relation $\sim_D$ partitions the set of labelings of $D$. We call the sets under this partition the \textit{classes of labelings}.

 Given a labeling $D_L$ of a diagram $D$, associate to it a poset $P_L^D$ on the boxes of the diagram, which restricts to a total order of the boxes of $D$ in the same column or same row, as prescribed by the labeling $D_L$, and in which these are all of the relations. The following lemma is a direct consequence of the definitions.
 
 \bl \label{posets:classes} Given a diagram $D$ and two labelings $D_L$ and $D_{L'}$ of it,  $D_L \sim_D D_{L'}$ if and only if the posets $P_L^D$ and $P_{L'}^D$ are equal.
 \el

While the next Lemma is also relatively straightforward, the idea of its proof is repeatedly used in this paper. 

\bl \label{lemma:classes} Let $\l=(v+1, 1^{l-1}) \in \Dk$ and $D \in \Dk$ be a forest with $l+v$ boxes and at least $l$ rows and $v+1$ columns.  Then the following two sets are equal:

1.  the classes of labelings of $D$ such that the class contains a labeling with:

  $i_1, \ldots, i_l$ are distinct, $j_1\leq \cdots\leq j_l$, $j_{l+1}, \ldots, j_{l+v}$ are distinct, $i_{l+1} \leq \cdots\leq i_{l+v}$ 

2.  the classes of labelings of $D$ such that  the class contains a labeling with:

$i_1, \ldots, i_{l-1}$ are distinct, $j_1\leq \cdots\leq j_{l-1}$, $j_{l}, \ldots, j_{l+v}$ are distinct, $i_{l} \leq \cdots\leq i_{l+v}$ 

\el

Note that the condition that $\l=(v+1, 1^{l-1}) \in \Dk$ signifies that $k\geq l$ and $n-k \geq v+1$. Also, as seen from the requirement on the forests $D$ we consider, the number of boxes in $D$ is the same as the number of boxes in $\lambda$.  We say that a forest $D$ can be {\it labeled with respect to $\lambda$}, or that a labeling of a forest $D$ is with respect to $\lambda$, if the number of boxes of $\lambda$ and $D$ are the same, the number of rows and columns of $D$ are at least as many as those of $\lambda$ and if there is a labeling of $D$ as prescribed by condition 1 (or 2) in Lemma  \ref{lemma:classes}. Moreover, a {\it class of labelings with respect to $\lambda$} is a class of labelings which contains a labeling with respect to $\lambda$. 
We refer to a labeling that satisfies condition 1 in Lemma \ref{lemma:classes} as a labeling from {\it class 1}, and a labeling that satisfies condition 2 in Lemma \ref{lemma:classes} as a labeling from {\it class 2}.

\medskip

\begin{figure}
\begin{center}
\includegraphics[scale=1.2]{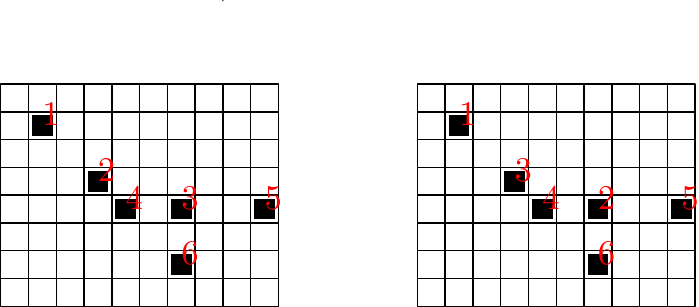}
 \caption{The black boxes indicate the forests in the two $8\times 10$ rectangles.  The red numbers signify the labelings of the forests. Let $\lambda=(4, 1^2)$. The left hand side labeling $L$ is a labeling from class 1 and is equivalent to the right hand side labeling $L'$, which is from class 2. $L'$ is constructed from $L$ as described in the proof of Lemma \ref{lemma:classes}.}
 \label{lemma}
 \end{center}
\end{figure}

\noindent{\it Proof of  Lemma \ref{lemma:classes}.}
We need to show that for every monomial $x^{D_L}$, where $L$ is a labeling in one of the classes, there is a monomial $x^{D_{L'}}$, such that $L'$ is a labeling from the other class and $x^{D_L}=x^{D_{L'}}$. 

Let $L$ be a labeling from class 1, i.e. $x^{D_L} = x_{i_1 j_1}\ldots x_{i_{l+v}j_{l+v}}$ with $j_1\leq \cdots \leq j_l$ and $i_1,\ldots,i_l$ distinct and $i_{l+1}\leq \cdots \leq i_{l+v}$ and $j_{l+1},\ldots,j_{l+v}$ distinct. 
Since $D$ has at least $v+1$ columns, there is an index $r \leq l$, such that $j_r \not \in \{j_{l+1},\ldots,j_{l+v}\}$. Let $r\leq l$ be the largest such index. Then $i_r \neq i_{r+1},\ldots,i_l$ and $j_r \neq j_{r+1},\ldots,j_l$, so $x_{i_rj_r}$ commutes with the variables at positions $r+1,\ldots ,l$ and can be moved to a position $r'-1 \geq l$, such that $r'$ is the smallest index greater than $l$ for which $i_r \leq i_{r'}$. Then 
$$x^{D_L} = x_{i_1j_1}\ldots x_{i_{r-1}j_{r-1}}x_{i_{r+1}j_{r+1}}\ldots x_{i_lj_l} \ldots x_{i_rj_r}x_{i_{r'}j_{r'}}\ldots$$
and since $j_r$ is different from any of $j_{l+1},\ldots,j_{l+v}$ the last monomial is a labeling in  class 2. For an example see Figure \ref{lemma}.

The case when $L$ is a labeling from class 2 follows the same reasoning by exchanging the roles of $i$ and $j$.
\qed

Let  $\L_1^{D, \l}, \ldots, \L_m^{D, \l}$ be all the classes of labelings of a forest $D$ with respect to $\l$ (see definition after Lemma  \ref{lemma:classes}). Let $D_{L_i} \in \L_i^{D, \l}$, $i \in [m]$, be  (arbitrary)  representative labelings from those classes. Denote by $\L(D, \l)=\{D_{L_1}, \ldots, D_{L_m}\}$ these representative labelings. 

\begin{theorem} \label{thm:hooks} Let $\l=(s, 1^{t-1})$ be a hook that fits in a $k\times (n-k)$ rectangle. Then, 
\begin{equation} \label{exp} \S_{w(\l,k)}(\t_1,\ldots,\t_n)=s_{\l}(\t_1, \ldots, \t_k)=\sum_{D \in \D_{k\times(n-k)}} c_D^\l {\sum_{D_L \in \L(D, \l)}x^{D_L}}, \end{equation}
where   
\begin{equation} 
\label{cdl} c_D^\l={{\r(D)-t+\c(D)-s}\choose{\c(D)-s}},\end{equation} 
if  for the forest $D$ we have $\r(D)\geq t, \c(D)\geq s$, and otherwise $c_D^\l=0$. 


\end{theorem}

\noindent {\bf Remark.} The coefficient $c_D^\l$ in  Theorem \ref{thm:hooks} is equal to the multiplicity of the Specht module $S^\lambda$ in the Specht module $S^D$ (when $D$ is a forest) which can be seen, as  Liu \cite{liu2} pointed out, as a consequence of   \cite[Theorem 4.2]{liu}. This appears to be a coincidence, though it would be amazing to discover a conceptual connection between the expansion \eqref{exp} and representations of the symmetric group.

\medskip

Before proceeding to the proof of Theorem \ref{thm:hooks} we state a few lemmas which we use in it.

\noshow{ We leave the proof of Lemma \ref{classes} to the interested reader.}

 \bl \label{lemma:rectangle_fit}
 Let $\l$ be a partition that does not fit into a $a\times b$ rectangle. Then, $$s_\l(\t_1, \ldots, \t_a)=0 \text{ in } \mathcal{E}_{a+b}.$$
 \el
 
 \proof 
 The statement follows readily from Theorem \ref{th:pieri} for elementary and homogeneous symmetric functions, namely $e_k(\t_1, \ldots, \t_a)=0 \text{ and } h_m(\t_1,\ldots,\t_b)=0 \text{ in } \mathcal{E}_{a+b}$ for $k>a$ and  $m>b$.
 Using the Jacobi-Trudi determinant expansion and its dual for any Schur function,
 $$s_{\lambda}=\det [ h_{\lambda_i-i+j}]_{i,j=1}^n = \det[ e_{\lambda'_i -i+j}]_{i,j=1}^n,$$
 we see that if $\lambda_1>b$  or $\lambda'=l(\lambda)>a$ the top row of the first matrix or the first column of the second, and hence the determinant, is 0.   \qed
 
 \begin{corollary} \label{kill} $e_a h_b(\t_1, \ldots, \t_a)=0 \text{ in } \mathcal{E}_{a+b}.$
 \end{corollary}
 \proof By the Pieri rule $e_a h_b=s_{(b+1, 1^{a-1})}+s_{(b, 1^{a})}$, and the shapes $(b+1, 1^{a-1})$ and $(b, 1^{a})$ do not fit into a $a\times b$ rectangle.
 \qed
 
 \medskip
 
Next we consider several induced objects in the rectangle $R_{k \times (n-k)}$. Namely, for $\{i_1, \ldots, i_a\} \subset \{1, \ldots, k\}$, $\{j_1, \ldots, j_b\} \subset \{k+1, \ldots, n\}$, with $|\{i_1, \ldots, i_a\}|=a$ and $|\{j_1, \ldots, j_b\}|=b$  we call $[i_1, \ldots, i_a]\times [j_1, \ldots, j_b]$, which denotes the squares in the intersection of a row indexed by $i_l$ and $j_m$, $l \in [a]$, $j \in [b]$,  an \textit{induced $a\times b$ rectangle}. Furthermore,   $e_a^{i_1, \ldots, i_a}=e_a(x_{i_1}, \ldots, x_{i_a})$ is the \textit{induced elementary symmetric function} and $h_b^{j_1, \ldots, j_b}=h_b(x_{j_1}, \ldots, x_{j_b})$ is the \textit{induced homogeneous symmetric function} and  $\mathcal{E}_{a+b}^{[i_1, \ldots, i_a]\times [j_1, \ldots, j_b]}$ the \textit{induced Fomin-Kirillov algebra} in the natural way, with $\t_l^{[i_1, \ldots, i_a]\times [j_1, \ldots, j_b]}$, $l \in [a]$, being the induced Dunkl element. With the above notation we can restate Corollary \ref{kill} as follows.

  \begin{corollary} 
\label{kill-induced} 
We have
$e_a^{i_1, \ldots, i_a} h_b^{j_1,
\ldots, j_b}(\t_1^{[i_1, \ldots, i_a]\times [j_1, \ldots, j_b]}, \ldots,
\t_a^{[i_1, \ldots, i_a]\times [j_1, \ldots, j_b]})=0$ in
$\mathcal{E}_{a+b}^{[i_1, \ldots, i_a]\times [j_1, \ldots, j_b]}$.
\end{corollary}

\medskip

\noindent \textit{Proof of  Theorem \ref{thm:hooks}.} We proceed by induction on the number of columns $\c(\l)$ of $\l$. When $\c(\l)=1$ the statement was given in Theorem \ref{th:pieri}. Assume that the statement is true for $\c(\l) \leq v$. We prove that it  is also true for all hooks $\l$ with $\c(\l)=v+1$. To do this we use Pieri's rule:  \begin{equation} \label{pieri} e_l h_v=s_{(1^l)}h_v=s_{(v+1, 1^{l-1})}+s_{(v, 1^{l})}.
\end{equation}

Let $\l=(v+1, 1^{l-1})$ and $\b=(v, 1^{l})$.
If we evaluate equation \eqref{pieri} at $\t$ and expand $e_l$ and $h_v$ according to \cite[Theorem 3.1]{P} we obtain  
\begin{equation} 
(\sum_{\substack{i_1, \ldots, i_l \neq \\ j_1\leq \cdots \leq j_l}} x_{i_1j_1} \cdots x_{i_lj_l})(\sum_{\substack{i_{l+1}\leq \cdots \leq i_{l+v} \\ j_{l+1}, \ldots, j_{l+v} \neq}} x_{i_{l+1}j_{l+1}} \cdots x_{i_{l+v}j_{l+v}})=s_{\l}(\t)+s_{\b}(\t)
\end{equation}

and we want to prove that 

\begin{equation} \label{ad}
\begin{array}{l}
\displaystyle
(\sum_{\substack{i_1, \ldots, i_l \neq \\ j_1\leq \cdots \leq j_l}} x_{i_1j_1}
\cdots x_{i_lj_l})(\sum_{\substack{i_{l+1}\leq \cdots \leq i_{l+v} \\ j_{l+1},
\ldots, j_{l+v} \neq}} x_{i_{l+1}j_{l+1}} \cdots x_{i_{l+v}j_{l+v}})\\[.4in]
\displaystyle
\qquad\qquad =\sum_{D \in \D_{k\times(n-k)}} (c_D^\l (\sum_{D_L \in \L(D, \l)}x^{D_L})+c_D^\b
(\sum_{D_L \in \L(D, \b)}x^{D_L})). 
\end{array} 
\end{equation}

Given the properties of $c_D^\l, c_D^\b$ and $\L(D, \l), \L(D, \b)$ (in light of Lemma \ref{lemma:classes}) we can rewrite \eqref{ad} as 

\begin{equation} \label{ad1}
\begin{array}{l}
\displaystyle
(\sum_{\substack{i_1, \ldots, i_l \neq \\ j_1\leq \cdots \leq j_l}} x_{i_1j_1}
\cdots x_{i_lj_l})(\sum_{\substack{i_{l+1}\leq \cdots \leq i_{l+v} \\ j_{l+1},
\ldots, j_{l+v} \neq}} x_{i_{l+1}j_{l+1}} \cdots x_{i_{l+v}j_{l+v}})
\\[.4in]
\displaystyle
\qquad\qquad =
\sum_{D
\in \D_{k\times(n-k)}} (c_D^\l +c_D^\b)(\sum_{D_L \in \L(D, \l) \cup  \L(D, \b)
}x^{D_L}), 
\end{array}
\end{equation}
where for the forests $D$ which have at least $v+1$ columns and $l+1$ rows, and
which can be labeled with respect to $\l$ and $\b$ as prescribed by Lemma
\ref{lemma:classes}, we pick the same representative labelings in $ \L(D, \l)$
and $\L(D, \b)$.

Then, if forest $D$  has exactly  $v$ columns or $l$ rows, but can be labeled with respect to $\b$ or $\l$, respectively,  as prescribed by Lemma \ref{lemma:classes}, we have that $c_D^\l +c_D^\b=1$. If on the other hand we have a labeling $D_L \in  \L(D, \l) \cap  \L(D, \b)$, then 
using \eqref{cdl} we obtain that 
\begin{align}
c_D^\l+c_D^\b&={{\r(D)-l+\c(D)-(v+1)}\choose{\c(D)-(v+1)}}+{{\r(D)-(l+1)+\c(D)-v}\choose{\c(D)-v}}\\ &={{\r(D)+\c(D)-(l+v)}\choose{\c(D)-v}}={{c(D)}\choose{\c(D)-v}},
\end{align}

where $c(D)$ denotes the number of components of $D$.

Thus we can rewrite \eqref{ad1} as 
\be \label{ad2}
(\sum_{\substack{i_1, \ldots, i_l \neq \\ j_1\leq \cdots \leq j_l}} x_{i_1j_1} \cdots x_{i_lj_l})(\sum_{\substack{i_{l+1}\leq \cdots \leq i_{l+v} \\ j_{l+1}, \ldots, j_{l+v} \neq}} x_{i_{l+1}j_{l+1}} \cdots x_{i_{l+v}j_{l+v}})=\ee
\be \label{ad3}
\sum_{D \in \D_{k\times(n-k)}} \left({{c(D)}\choose{\c(D)-v}}(\sum_{D_L \in \L(D, \l) \cap  \L(D, \b) }x^{D_L})+ (\sum_{D_L \in \L(D, \l) \triangle  \L(D, \b) }x^{D_L})\right ).
\ee

We now show that the coefficient of $x^{D_L}$, $D_L \in \L(D, \l) \cup  \L(D, \b)$, is the same  in  \eqref{ad2} and \eqref{ad3}, and that the remainder of the terms in \eqref{ad2} sum to zero, thereby proving the equality of \eqref{ad2} and \eqref{ad3}.

Consider first the case that $D_L \in \L(D, \l) \triangle  \L(D, \b)$. Then the coefficient of $x^{D_L}$ in   
 \eqref{ad3} is $1$ and the forests $D$ are such that $D$  has exactly  $v$ columns or $l$ rows, but can be labeled with respect to $\b$ or $\l$, respectively,  as prescribed by Lemma \ref{lemma:classes}. It is not hard to see then that the coefficient of $x^{D_L}$ (considered modulo commutations) in   
 \eqref{ad3} is also $1$.

Consider the case that $D_L \in \L(D, \l) \cap  \L(D, \b)$. Then the coefficient of $x^{D_L}$ in   
 \eqref{ad3} is ${{c(D)}\choose{\c(D)-v}}$ and the forests $D$ are such that $D$ has at least $v+1$ columns and $l+1$ rows, and $D$ can be labeled with respect to $\l$ and $\b$ as prescribed by Lemma \ref{lemma:classes}.  In order to calculate the coefficient of $x^{D_L}$ (considered modulo commutations) in    \eqref{ad2} we need to decide which variables of $x^{D_L}$ should come from $e_l$ (the first sum in  \eqref{ad2}) and which from $h_v$ (the second sum in  \eqref{ad2}) in  \eqref{ad2}. Considering variables as squares in the $k \times (n-k)$ rectangle, note that all but one square in each component of $D$ is a priori forced to be in $e_l$ or $h_v$ because of the conditions on the $i$'s and $j$'s, and this one square can go into either one. It is then easy to count how many squares are already assigned to $e_l$ (or $h_v$) and determine that we can pick   out exactly ${{c(D)}\choose{\c(D)-v}}$ terms  in   
 \eqref{ad2}  which are equal to $x^{D_L}$.

 It remains to show that  all the other terms on the left hand side sum to zero. This follows as  all the terms that are not of the form $x^{D_L}$, $D_L \in \L(D, \l) \cup  \L(D, \b)$ are part of a sum of terms which sum to zero as a consequence of Corollary \ref{kill-induced}.  \qed

\section{Action on the quantum cohomology}

\label{sec:hom}

Recall that $s_{ij}$ is the transposition of~$i$ and~$j$ in~$S_n$,
$s_i=s_{i\,i+1}$ is a Coxeter generator,
and $q_{ij}=q_i q_{i+1}\cdots q_{j-1}$, for $i<j$.
Define the $\Z[q]$-linear operators $t_{ij}$, $1\leq
i<j \leq n$, acting on the quantum cohomology ring $\QH^*(\Fl_n,\Z)$ by 
\begin{equation}
\label{eq:tij}
  t_{ij}(\s_w) =
  \left\{ 
    \begin{array}{ll}
      \s_{w s_{ij}} & \textrm{if } \l(w s_{ij})=\l(w)+1\,,\\[.05in]
      q_{ij}\, \s_{w s_{ij}} & \textrm{if } \l(w s_{ij})= \l(w) -2(j-i)+1
           \,,\\[.05in]
      0   & \textrm{otherwise.}
    \end{array}
  \right.
\end{equation}
By convention, $t_{ij}=-t_{ji}$, for $i>j$, and  $t_{ii}=0$.

\old{Quantum Monk's formula (Theorem~\ref{th:monk}) can be stated as 
saying that the quantum product of~$\s_{s_m}$ 
and~$\s_w$ is equal to
\[
  \s_{s_m}*\s_w=\sum_{a\leq m<b} t_{ab}(\s_w)\,.
\]
}

The relation between the algebra~$\E_n^q$ and quantum cohomology 
of~$\Fl_n$
is justified by the following lemma, 
which is proved by a direct verification.

\begin{lemma} {\rm \cite[Proposition~12.3]{FK}} \
The operators $t_{ij}$ given by~{\rm (\ref{eq:tij})} satisfy the
relations in the algebra~$\E_n^p$ 
with $x_{ij}$ replaced by $t_{ij}$, $p_{i\,i+1}=q_i$,
and $p_{ij}=0$, for $|i-j|\geq 2$,
\end{lemma}

Thus the algebra~$\E_n^q$ acts on~$\QH^*(\Fl_n,\Z)$ by 
$\Z[q]$-linear transformations
\[
  x_{ij}:\,\s_w\longmapsto t_{ij}(\s_w)\,.
\]

\old{Monk's formula is also equivalent to the claim that the Dunkl element
$\hat{\theta}_i$ acts on the quantum cohomology of~$\Fl_n$ as the operator of
multiplication by~$x_i$,  the latter is defined via the
isomorphism~(\ref{eq:q-factor}).

Let us denote $c(k,m)=s_{m-k+1}s_{m-k+2}\cdots s_m$ 
and $r(k,m)=s_{m+k-1}s_{m+k-1}\cdots s_m$.
These are two cyclic permutations such that 
$c(k,m)=(m-k+1,m-k+2,\dots,m+1)$ and $r(k,m)=(m+k,m+k-1,\dots,m)$.

The following statement was geometrically proved in~\cite{ciocan}
(cf.\ also~\cite{FGP}).  For the reader's convenience and for 
consistency we show how to deduce it directly from Monk's formula.
}

The following lemma follows directly from equations \eqref{schub_e} and \eqref{schub_e_q}. It is the key to showing that our nonnegative expansions of certain Schubert polynomials evaluated at the Dunkl elements imply that the same expansions are equal to the evaluation of the corresponding $p$-quantum Schubert polynomials  $ \frak{S}_w^p$  (and so in particular quantum Schubert polynomials  $ \frak{S}_w^q$) at the Dunkl elements.

\begin{lemma}
\label{obv} 
Suppose that the identity 
$$
f(\x)=F(f_1(\x), \ldots, f_k(\x)), 
$$
holds,  where $f$ and the $f_i$'s are Schubert 
polynomials and $F$ is a
polynomial in $k$ variables. Suppose that there are expansions of $f_i(\t)$ and
$f_i^p(\t)$ which are in $\E_n^+$ and are equal to each other. If the expansion
we obtain for $f(\t)$ by evaluating $F$ at the above mentioned expansions of
$f_i(\t)$'s is in $\E_n^+$ without involving the relation $x_{ij}^2=0$, then there is an identical expansion of $f^p(\t)$.
\end{lemma}

\begin{lemma}
Let $\l=(s,1^{t-1})$. The coset of the polynomial $s_{\lambda}(x_1,\dots,x_m;q)$ 
in the quotient ring~{\rm (\ref{eq:q-factor})} corresponds to the
Schubert class~$\s_{w(\l,k)}$ under the isomorphism~{\rm (\ref{eq:q-factor}).}
\end{lemma}

\old{\proof  By~(\ref{eq:q-isom}) and~(\ref{eq:E2}), it is enough to check that
\[
  \s_{c(k,m)}=\s_{c(k,m-1)}+(\s_{s_{m}}-\s_{s_{m-1}})*\s_{c(k-1,m-1)}
  + q_{m-1} \s_{c(k-2,m-2)}.
\]
This identity immediately follows from Monk's formula:
\[
  (\s_{s_{m}}-\s_{s_{m-1}})*\s_{c(k-1,m-1)} = 
  (\sum_{b>m} t_{mb} - \sum_{a<m} t_{am})(\s_{c(k-1,m-1)})\,.
\]
The claim about~$\s_{r(k,m)}$ can be proved using a symmetric argument.
\endproof
}

We can now use Lemma \ref{obv} and apply it to the steps of the proof of Theorem \ref{thm:hooks}, to see that it is also true in the $p$-quantum world: 

\begin{theorem} \label{thm-p} Let $\l=(s, 1^{t-1})$ be a hook that fits in a $k\times (n-k)$ rectangle. Then, 
\begin{equation} \S^p_{w(\la,k)}(\t_1,\ldots,\t_n)=s^p_{\l}(\t_1, \ldots, \t_k)=\sum_{D \in \D_{k\times(n-k)}} c_D^\l {\sum_{D_L \in \L(D, \l)}x^{D_L}}, \end{equation}
where   
\begin{equation} \label{cdl1} c_D^\l={{\r(D)-t+\c(D)-s}\choose{\c(D)-s}},\end{equation} 
if  $\r(D)\geq t, \c(D)\geq s \text{ and } D \text{ is a forest}$, and otherwise $c_D^\l=0$. 


\end{theorem}


 Theorem \ref{thm-p} and its proof together with Lemma \ref{obv} imply the following statement.

\begin{corollary} 
\label{cor:pieri-QH}  
  For any $w\in S_n$ the product of Schubert 
  classes~$\s_{w(\l,k)}$ ,  where $\la=(s,1^{t-1})$, 
  and~$\s_w$ in the quantum cohomology ring~$\QH^*(\Fl_n,\Z)$
  is given by the formula
  \begin{equation}
    \s_{w(\la,k)}*\s_w = \sum_{D \in \D_{k\times(n-k)}} c_D^\l {\sum_{D_L \in \L(D, \l)}t^{D_L}}(\s_w), \end{equation}
where   
\begin{equation} \label{cdl2} c_D^\l={{\r(D)-t+\c(D)-s}\choose{\c(D)-s}},\end{equation} 
if  $\r(D)\geq t, \c(D)\geq s \text{ and } D \text{ is a forest}$, and otherwise $c_D^\l=0$. 
    \end{corollary}

\section{ Nonnegativity Conjecture for $s_{\lambda}$ for other shapes $\lambda$}
\label{sec:other}

In this section we investigate the nonnegativity conjecture for Schubert polynomials of the form $s_{\lambda}(x_1,\ldots,x_k)$ for other shapes $\lambda$. Throughout this section $k$ will be fixed and we set $\t=(\t_1,\ldots,\t_k)$.

Consider first the shapes $\mu = (n-k)^r$ or $\nu= r^k$ which correspond via \eqref{d:perm_la} to Grassmannian permutations  $w(\mu,k)$ and $w(\nu,k)$. Applying Lemma \ref{lemma:rectangle_fit} and the Jacobi-Trudi identity it follows that $s_{\mu}(\t_1,\ldots,\t_k) = h_{n-k}(\t)^r$ and $s_{\nu}(\t_1,\ldots,\t_k)=e_k(\t)^r$. An obviously nonnegative expansion is an immediate consequence of the above and Theorem \ref{th:pieri}.

\begin{proposition}\label{prop:rectangles}
For any $k$, $r\leq k$ and $t\leq n-k$ let $\mu = (n-k)^r$ and $\nu= t^k$
we have the following expansions in $\E_n^+$ (in $\E_n^q$):
\begin{multline}
\S_{w(\mu,k)}(\t_1,\ldots,\t_k) = \left( \sum_{\substack{i_1\leq \cdots \leq i_k \leq k;\\ k+1\leq j_1,\ldots,j_k \neq} } x_{i_1j_1}\cdots x_{i_kj_k} \right)^r \\
\S_{w(\nu,k)}(\t_1,\ldots,\t_k) =  \left( \sum_{\substack{ k+1\leq j_1\leq \cdots \leq j_k; \\ k \geq i_1,\ldots,i_k \neq }} x_{i_1j_1}\cdots x_{i_kj_k} \right)^t,
\end{multline}
where the first sum goes over all sequences of $i$ and $j$ of length $k$, such that the $i$s are weakly increasing, $\leq k$, and the $j$s are $\geq k+1$ and all distinct; and in the second sum the $i$s are distinct and the $j$s increasing. 
\end{proposition}

We now focus on $s_{\lambda}$ where $\lambda$ is a hook plus a box at $(2,2)$. We  show that:
\begin{theorem}\label{tim:hook+box}
The Schubert polynomial $\S_{w(\la,k)}(\t_1,\ldots,\t_n)$, where $\la =(b,2,1^{a-1})$, has an expansion in $\mathcal{E}_n^{+}$. 
Equivalently,  $s_{(b,2,1^{a-1})}(\theta_1,\ldots,\theta_k) \in \mathcal{E}_n^+$.
\end{theorem}

\begin{proof}
To prove that $s_{(b,2,1^{a-1})}(\theta_1,\ldots,\theta_k) \in \mathcal{E}_n^+$ we use the Pieri rule:
\begin{align}\label{pieri_hookplus}
 s_{(b,2,1^{a-2})}= s_{(b,1^{a-1})}h_1 - s_{(b,1^a)}- s_{(b+1,1^{a-1})} .
\end{align}
Recall that  $h_1(\t)=s_{(1)}(\t)=\sum_{i\leq k, k<j} x_{ij}$.
The expansion for hooks in Theorem \ref{thm:hooks} gives us the following formulas for the three hooks in equation \eqref{pieri_hookplus}:
\begin{align}\label{3hooks}
s_{(b,1^{a-1})}(\t)&=\sum_{D \in \D_{k\times (n-k)}} \sum_{D_L \in \L(D,(b,1^{a-1}))} c_D^{(b,1^{a-1}) }x^{D_L}\\ \label{3hooks1}
s_{(b,1^a)} (\t)&=\sum_{D \in \D_{k\times (n-k)}} \sum_{D_L \in \L(D,(b,1^{a}))} c_D^{(b,1^{a}) }x^{D_L} \\ \label{3hooks2}
s_{(b+1,1^{a-1})}(\t)&=\sum_{D \in \D_{k\times (n-k)}} \sum_{D_L \in \L(D,(b+1,1^{a-1}))} c_D^{(b+1,1^{a-1})} x^{D_L}. 
\end{align}
We will consider the sequences of indices appearing in each monomial $x^{D_L}$ and 
 for $I=(i_1, \ldots, i_l) \in [1\ldots k]^l$, $J=(j_1, \ldots, j_l)\in [k+1 \ldots n]^l $ we define  $x_{IJ} = x_{i_1j_1}\cdots x_{i_l j_l}$. For each of the terms on the right hand side of  \eqref{3hooks}-\eqref{3hooks2} by Lemma \ref{lemma:classes} we can choose  sequences of indices $I$ and $J$ such that $x^{D_L}=x_{IJ}$ and $I=(I_1,I_2)$, $J=(J_1,J_2)$, where $I_1$ and $J_1$ are sequences of length $a$, the elements in $I_1$ and $J_2$ are distinct and the elements in $J_1$ and $I_2$ are weakly increasing. Notice also that the number of distinct rows in $D$ is the same as the number of distinct elements in $(I_1,I_2)$ and  the number of columns is the cardinality of $J$ as a set. 
 
It will be more convenient to express the coefficients $c_D^{\lambda}$ appearing in \eqref{3hooks}-\eqref{3hooks2}  in terms of the sequences of indices just considered. Here $|S|$ will denote the number of distinct elements of $S$.  
The coefficients in front of $x^{D_L} =x_{IJ}$ are given by
\begin{align}\label{coefficients_hookplusbox}
c_D^{(b,1^{a-1})} &= \binom{|I_1\cup I_2| + |J_1\cup J_2| - a-b}{|I_1\cup I_2|-a},\\
c_D^{(b,1^{a})} &= \binom{|I_1\cup I_2| + |J_1\cup J_2| - a-b-1}{|I_1\cup I_2|-(a+1)}, \\
c_D^{(b+1,1^{a-1})} &= \binom{|I_1\cup I_2| + |J_1\cup J_2| - a-b-1}{|I_1\cup I_2|-a} .
\end{align}

Notice that in  the expressions of the two hooks of size $a+b$, the lengths of the index sequences $I_1$ and $I_2$ are the same ($a$ and $b$, correspondingly), so we can combine the expressions as
\begin{multline}\label{rhs}
s_{(b,1^a)}(\t)+s_{(b+1,1^{a-1})}(\t) = \sum_{D\in \D_{k\times (n-k)}  }\sum_{\substack{ D_L\in \L(D,(b,1^a)),\\  x^{D_L}\sim_D x_{I_1J_1}x_{I_2J_2}}}\\
\left( \binom{|I_1\cup I_2| + |J_1\cup J_2| - a-b-1}{|I_1\cup I_2|-a} + \binom{|I_1\cup I_2| + |J_1\cup J_2| - a-b-1}{|I_1\cup I_2|-(a+1)} \right) x^{D_L}\\
= \sum_{D\in \D_{k\times (n-k)} } \sum_{\substack{D_L\in \L(D,(b,1^a)),\\  x^{D_L}\sim_D x_{I_1J_1}x_{I_2J_2}}} 
\binom{|I_1\cup I_2| + |J_1\cup J_2| - a-b}{|I_1\cup I_2|-a} x^{D_L},
\end{multline}
where the sum  goes over all diagrams (which are forests) in the $k\times n-k$ rectangle and $D_L$ goes over all labeling classes in $\L(D,(b,1^a))$ and $I_1,J_1,I_2,J_2$ are sequences of indices, such that $x_{I_1J_1}x_{I_2J_2}$ is a representative of its class (see Lemma \ref{posets:classes}), and $I_1,J_1$ have $a$ elements and $I_2,j_2$ have $b$ elements. Since all diagrams considered in this proof are in $\D_{k \times (n-k)}$ summation over $D$ or $D'$ will mean summation over all diagrams in $\D_{k \times (n-k)}$.

We can write a similar expression for $s_{(b,1^{a-1})}(\t)$ with labelings $x^{D_L} \sim_D x_{I_1J_1}x_{I_2J_2}$ such that $I_1$ and $J_1$ have lengths $a$
\begin{multline}\label{lhs}
s_{(b,1^{a-1})}(\t)h_1(\t) = \sum_{i=1}^{k}\sum_{j=k+1}^{n} \sum_{D'} \\ \sum_{\substack{ L' \in \L(D',(b,1^{a-1})), \\ x^{L'} \sim_D x_{I_1J_1}x_{I'_2J'_2}}}
\binom{|I_1\cup I'_2| + |J_1\cup J'_2| - a-b}{|I_1\cup I'_2|-a} x_{I_1J_1}x_{I'_2J'_2}x_{ij},
\end{multline}
where the sum goes over all diagrams $D'$ and labeling classes $L'$ in $\L(D',(b,1^{a-1}))$, such that $x_{I_1J_1}x_{I'_2J'_2}$ is a class representative and the length of the sequences $I_1$ and $J_1$ is $ a$ and the length of  $I'_2$ and $J'_2$ is $b-1$.

For each monomial in \eqref{rhs} we will compare the coefficients with the corresponding coefficients in \eqref{lhs} and show that the ones in \eqref{rhs} are always smaller. 
Consider a monomial (in the $x$--variables) in \eqref{lhs} and consider its last variable $x_{ij}$, so the monomial can be written as
$x_{I_1J_1}x_{I_2J_2}=x_{I_1J_1}x_{I'_2J'_2}x_{ij}$, where $I_2 = (I'_2,i)$ and $J_2=(J_2',j)$. Clearly this term appears exactly like this in \eqref{lhs}.
Consider the difference $s_{(b,1^{a-1})}(\t)h_1(\t) - s_{(b,1^a)}(\t)-s_{(b+1,1^{a-1})}(\t)$.
The coefficient in front of $x_{IJ}x_{ij}$ (without involving any commutativity relations in $s_{(b,1^{a-1})}(\t)h_1(\t)$)  for $I=(i_1,I_2')$ and $J=(J_1,J_2')$ is
\begin{multline}\label{diff_coef}
\binom{ |I_1\cup I_2'| + |J_1\cup J_2'| -a -b}{|I_1\cup I_2'|-a} - \binom{ |I_1\cup I_2' \cup \{i\} | + |J_1\cup J_2' \cup \{j\}| -a-b}{|I_1\cup I_2' \cup \{i\} |-a}.
\end{multline}

Let  $A = |I_1 \cup I_2'|-a$ and $B=|J_1\cup J_2'|-b$.

There are $4$ different cases depending on whether $i \in I_1\cup I_2'$ and $j\in J_1\cup J_2'$, which we consider separately. In all these cases we  show that the total coefficient of  terms $\sim x_{IJ}x_{ij}$ is greater in \eqref{lhs} than in \eqref{rhs}, where $\sim$ means equivalence under commutation.

\textit{First case:} If $i \in I_1 \cup I_2'$ and $j \in J_1\cup J_2'$ then the coefficient in \eqref{diff_coef} is 0, so the total coefficient in front of $x_{IJ}x_{ij}$ is nonnegative.

For the other 3 cases we need to consider in how many  ways a monomial $x^{L'}x_{ij}$ appears in $s_{(b,2,1^{a-2})}(\t)h_1(\t)$ by applying the commutation relation to $x_{ij}$ and the remaining variables in $x_{IJ}$.

The $x$'s which could be moved to the end of $x_{IJ}$ by commutation are:
1) The ones in $x_{I'_2J'_2}$ which are last in a sequence of equal $i$s, so their index set is $(I_b,J_b)$, where $I_b$ is the set of all distinct elements in $I'_2$.
2) The ones in $x_{I_1J_1}$ which are last in a sequence of equal $j$s, $(I_a,J_a)$, such that $J_a$ is the set of distinct elements of $J_1$. Moreover, we can pick only these $x$'s, whose indices are not in $I_2'\cup J_2'$.

Once such an $x_{i_r,j_r}$ has been moved to the end, we can move $x_{ij}$ by commutation within $x_{I_2',J_2'}$ (without $x_{i_r,j_r}$) if $ i\neq i_r, j\neq j_r$, which gives a representative labeling class as in Lemma \ref{lemma:classes} (depending where we took $x_{i_rj_r}$ from): since $x_{I_1J_1}x_{I'_2J'_2}x_{ij}$ was a representative labeling for the hooks from \eqref{rhs}, we have that $j \not \in J_2'$ and thus $J_2' \cup \{j\}$ still has all $j$s distinct.  

Thus the number of  $x$'s we can move to the end (and insert $x_{ij}$) is:
\begin{multline}\label{comm_terms}
|I_2'\setminus \{i\}| + |(I_a,J_a)\setminus (I_2',J_2') \setminus \{i,j\}| \geq \max( |I_2'\setminus \{i\}|, |J_1 \setminus \{j\} \setminus J_2'|-1),
\end{multline}
where $(I_a,J_a)\setminus (I_2',J_2) = \{(i',j') \in (I_a,J_a), i' \not \in I_2', j' \not \in J_2'\}$ and so 
$ |(I_a,J_a)\setminus (I_2',J_2') \setminus \{i,j\}| \geq |(I_a,J_a) \setminus \{i,j\}| - |I_a \cap I_2'| - |J_a \cap J_2'|= |J_1\setminus \{j\}| - | I_a \cap (I_2' \cup \{i\})| -|J_1 \cap J_2'| $.

\textit{Second case:} 
If $i \in I_1 \cup I_2'$ and $j \not \in J_1 \cup J_2'$, then the difference \eqref{diff_coef} is
$$-\binom{A+B}{A-1},$$
assuming that $A \geq 1$, since otherwise we get 0 and there is nothing more to prove.

For each of the variables $x_{i'j'}$ that we take from $x_{IJ}$ and move to the end through commutation and insert $x_{ij}$ we get a commutation equivalent monomial $x_{I'J'}x_{i'j'}$ such that $x_{I'J'}$ is a valid labeling class. The coefficient $c$ of $x_{I'J'}x_{i'j'}$ in \eqref{lhs}, i.e., the coefficient of $x_{I'J'}$ in the expansion of $s_{(b,1^{a-1})}(\t)$, is at least
$$\binom{ |I_1\cup I_2'|-1 +|J_1\cup J'_2| -(a+b)}{|I_1\cup I_2'|-1-a}=
\binom{A+B}{A-1} \frac{B+1}{A+B}. $$
The number of variables $x_{i'j'}$ we can move to the end is given by \eqref{comm_terms} and is at least $|I_2'|-1\geq A-1$ and not less than 1, so the total coefficient at the commutation class $\sim x_{IJ}x_{ij}$  is at least
\begin{align*}
\noshow{\binom{|I_1\cup I_2'| + |J_1\cup J_2'| -a -b}{|I_1\cup I_2'|-a-1} 
\frac{\max( |I_2'\setminus \{i\}|, |J_1\setminus J_2'|-1)(|J_1\cup J_2'| -b+1)}{|I_1\cup I_2'| + |J_1\cup J_2'| -a -b}}
\binom{A+B}{A-1}\frac{\max(A-1,1)(B+1)}{A+B}\geq \binom{A+B}{A-1},
\end{align*}
\noshow{We have that $|J_1 \cup J_2'| - (b-1) = |J_1 \setminus J_2'|$ since $|J_2'|=b-1$ and thus the ratio is 
$$ \frac{\max( |I_2'\setminus \{i\}|, |J_1\setminus J_2'|-1)(|J_1\setminus J_2'| )}{|I_1\cup I_2'| + |J_1\setminus J_2'| -a -1}
\geq \frac{(|I_1\cup I_2'|-a-1)(|J_1\setminus J_2'|)}{|I_1\cup I_2'| + |J_1\setminus J_2'| -a -1} \geq 1,$$
if $|I_1\cup I_2'|-a-1 \geq 1$, $|J_1\setminus J_2' \geq 1$ and not both equal to 1, in which case the total coefficient is nonnegative.}
since $B\geq 0$ and $A \geq 1$. So the total coefficient of $x_{IJ}x_{ij}$ (under commutation)  is nonnegative in this case as well.
\noshow{We will now consider the extreme cases.

If $|J_1\cup J_2'|=b-1$, then on the one side we have the negative contribution of
$\binom{|I_1\cup I_2'| + |J_1\cup J_2'| -a -b}{|I_1\cup I_2'|-a-1} = 1$ and on the other side the binomial coefficient is 0, but in fact the actual coefficient is at least 1, since we can find at least one term to commute out with $x_{ij}$ unless $I_2'=\{i\}$. But then we go back to the original equation for the coefficient, \eqref{diff_coef}, and we see that it is $\geq 1 -1 =0$.

If $|I_1\cup I_2'|-a-1 = -1$, we again go back to \eqref{diff_coef} to get that the coefficient is at least 0. If $|I_1\cup I_2'|-a-1 =0$ (and, presumably $|J_1\cup J_2'|> b-1$, so that $J_1 \not \subset J_2'$, then we can again find an $x$ to commute out: either $I_2' \neq \{i\}$, so there is one $x$ to get out of there, or else $I_2'=\{i\}$ and $i \not \in I_1$, so we can commute to the back an element with $x_{i_rj_r}$ with $j_r \in J_1 \setminus J_2'$.

If $|J_1\cup J_2'|=b$, then the coefficient is $-(|I_1\cup I_2'| - a)$ and the elements we can commute out are at least $|I_2' \setminus \{i\}| \geq  |I_1\cup I_2'| - a$ unless $I_2'\cap I_1 = \emptyset$ and $i \in I_2'$. But then we must be able to also commute one of the elements in $I_1,J_1\setminus J_2'\neq \emptyset$.
}

\textit{Third case:}
Let  $i \not \in I_1\cup I_2'$, but $j \in J_1\cup J_2'$. 
The coefficient in front of $x_{IJ}x_{ij}$ (without involving any commutation) is given in \eqref{diff_coef} as
\begin{multline*}\binom{ |I_1\cup I_2'| + |J_1\cup J_2'| -a -b}{|I_1\cup I_2'|-a} - \binom{ |I_1\cup I_2' \cup \{i\} | + |J_1\cup J_2' \cup \{j\}| -a-b}{|I_1\cup I_2' \cup \{i\} |-a} \\=  -\binom{ A+B}{ A+1}.
\end{multline*}
Consider the elements in $(I_a,J_a)$ and $(I_b,J_b)$ which we can  move to the end by commutation. As in the second case, for each  variable we move to the end (and insert $x_{ij}$) we get a coefficient coming from the expansion of $s_{(b,1^{a-1})}(\t)$ of at least
$$\binom{ A+B-1}{A+1} = \binom{A+B}{A+1}\frac{A+1}{A+B}.$$
The number of such variables we can move is at least, by \eqref{comm_terms}, $ \max( A, B-1)$. So the total coefficient is at least
$$\binom{A+B}{A+1}\frac{(A+1)\max(B-1,A)}{A+B} \geq \binom{A+B}{A+1}$$
and the coefficient of $x_{IJ}x_{ij}$ is again nonnegative. 

\noshow{Again we consider the elements in $(I_a,J_a)$ and $(I_b,J_b)$ which we can commute out. Again, for each such term we take out (and insert $x_{ij}$) we get a coefficient of at least
\begin{multline*}\binom{ |I_1\cup I_2'| + |J_1\cup J_2'|-1 -a -b}{|I_1\cup I_2'|-a} \\
= \binom{ |I_1\cup I_2'| + |J_1\cup J_2'| -a -b}{|I_1\cup I_2'|-a+1}\frac{|I_1 \cup I_2'|-a+1}{|I_1\cup I_2'| + |J_1\cup J_2'| -a -b}.\end{multline*}
and the number of such terms is at least $ \max( |I_2'|-1, |J_1\setminus J_2'|-1$. So the total coefficient is at least
$$\binom{ |I_1\cup I_2'| + |J_1\cup J_2'| -a -b}{(|I_1\cup I_2'|-a+1)}\frac{((|J_1\cup J_2'|-b)(|I_1 \cup I_2'|-a+1)}{|I_1\cup I_2'| + |J_1\cup J_2'| -a -b}$$
The factor is at least 1 unless $|J_1\cup J_2'|-b \leq 0$. But if $|J_1\cup J_2'|\leq b$, the binomial coefficient above is 0 and there is nothing left to prove. }

\textit{Fourth case:} Finally, let $i \not \in I_1\cup I_2'$ and $j \not \in J_1 \cup J_2'$. 
Then if we move any $x$ to the end by commutation and insert $x_{ij}$, we are not decreasing the number of rows or columns in $D$.  In \eqref{diff_coef} we have
$$\binom{|I_1\cup I_2'|+|J_1\cup J_2'|-a-b}{|I_1\cup I_2'|-a} - 
\binom{|I_1\cup I_2'|+|J_1\cup J_2'|-a-b+2}{|I_1\cup I_2'|-a+1}.$$
The number of terms that can be moved to the end by commutation is at least 
$\max(|I_2'|, |J_1 \cup J_2'|-(b-1))\geq \max(A, B+1)$. The coefficient of $x_{IJ}x_{ij}$ (under commutation) is at least 
\begin{multline*}
(\max(A,B+1)+1)\binom{A+B}{A} -\binom{A+B+2}{A+1}=\\
\frac{(A+B)!}{A!B!}(\max(A,B+1)+1 - \frac{(A+B+1)(A+B+2)}{(A+1)(B+1)})\geq 0,
\end{multline*} 
whenever $A\geq 0, B>1$. This expression is less than $0$ only if $B=1$ and $A\leq 2$ or $B=0$. But in each of these cases a more careful analysis of what elements can be moved out shows again that the coefficient of $x_{IJ}x_{ij}$ (under commutation) is nonnegative and this completes the proof.  
\noshow{Let $|J_1\cup J_2'|=b+1$. Since we can commute out at least $|I_2'|$ terms we have a total coefficient of at least $(|I_2'|+1)\binom{|I_1\cup I_2'|-a +|J_1\cup J_2'|-b}{|I_1 \cup I_2'|-a}=(|I_2'|+1)(|I_1\cup I_2'|-a+1)$. On the RHS we have
$\binom{|I_1\cup I_2'|-a +|J_1\cup J_2'|-b+2}{|I_1 \cup I_2'|-a+1}=\frac{(|I_1 \cup I_2'|-a+2)(|I_1 \cup I_2'|-a+3)}{2}$. The LHS is at least as large as the RHS, unless $|I_2'|=1$, but then we can commute out elements from the first part....
If $|J_1\cup J_2'|=b$ then the RHS is $|I_1 \cup I_2'|-a+2$, while the LHS is at least $(|I_2'|+1)$ and this is smaller only when $|I_2'|\leq |I_1\cup I_2'|-|I_1|=|I_2'\setminus I_1|$, so we must have $I_2' \cap I_1=\emptyset$, but then we can commute at least one more element from the first part (the one that is not in $J_2'$), so we are nonnegative again.}
\end{proof}

We can now use Lemma \ref{obv} and apply it to the steps of the proof of Theorem \ref{tim:hook+box}, to see that it is also true in the $p$-quantum world:

\begin{theorem}
The quantum and $p$-quantum Schubert polynomials  $\S_{w_b}^q$ and $\S_{w_b}^p$, where $w_b=w((b,2,1^{a-1}),k)$,  have  expansions in $\mathcal{E}_n^{+}$.
\end{theorem}

\medskip

While an explicit expansion for any general shape other than the hook remains elusive so far, we can derive such an expansion for the simplest case of a hook plus a box, namely, for $\lambda=(2,2)$ corresponding to $\S_w$, where $w_{k-1}=k+1,w_{k}=k+2,w_{k+1}=k-1,w_{k+2}=k$ and $w_i=i$ otherwise.

\begin{theorem}\label{thm:twobytwo} The Schubert polynomial $\S_w$ for $w=w((2,2),k)$ and its quantum version $\S_w^q$ have the following expansion in $\E_n^+$:
$$\S_w(\t_1,\ldots,\t_k)=s_{(2,2)}(\t_1,\ldots,\t_k) = \sum_{L: x_L \sim x_{IJ} } c_{IJ} x_{IJ},$$
where the sum runs over all classes $x_L \sim x_{IJ}$ distinct under commutation of the variables in $x_{IJ}$ and the coefficients are given by:
$$c_{IJ}=\begin{cases} 2, & \text{ if }|I|=|J|=4, \\ 0, & \text{ if } I \text{ or }J\text{ have an index of multiplicity 3 or 4},\\ 0, & \text{ if } x_{IJ}\sim x_{aj_1}x_{bj_1}x_{bj_2}x_{cj_2}, \text{ or }x_{IJ} \sim x_{i_1a}x_{i_1b}x_{i_2b}x_{i_3c}, \\
1, & \text{ otherwise. }
\end{cases}$$
Thus in the quantum cohomology ring $\QH^*(\Fl_n,\mathbb{Z})$ we have
$$\sigma_{w}*\sigma_{\pi} = \sum_{(I,J)} c_{IJ} t_{IJ}(\sigma_{\pi}).$$
\end{theorem}

\begin{proof}

We employ the notation from the previous proof, where for sequences of indices $I=(i_1,\ldots)$ and $J=(j_1,\ldots)$, we set $x_{IJ}=x_{i_1j_1}x_{i_2j_2}\cdots$.
Here we determine the coefficient of $x_{IJ}$, where $x_{IJ}$s are considered up to commutation. In other words, if $x_{I'J'}$ can be obtained from $x_{IJ}$ only by using the commutation relation, then these terms are considered equivalent. 
Let $c_{IJ}$ be the coefficient of $x_{IJ}$ in the expansion of $s_{(2,2)}$. We will denote by $[x]f$ the coefficient of $x$ in $f$ and $f|_{I}$ the restriction of $f$ to its summands whose first indices are in $I$.

The Jacobi-Trudi identity gives the following expressions $$s_{(2,2)}=h_2h_2 - h_3h_1=e_2e_2-e_3e_1.$$ 

Monomials with first indices $i$ coming from a given fixed set $\mathcal{I}$ can be obtained by restriction of the evaluation to the corresponding sets of indices.  Every function we consider here is expressed through the elementary and homogenous symmetric functions whose expansions can be restricted to any sets of first or second indices. Thus when $\# \mathcal{I} =1$  we have $e_2(\t)|_{\mathcal{I}}=0$ and $e_3(\t)|_{\mathcal{I}}=0$, so $s_{(2,2)}(\t)|_{\mathcal{I}}=0$ and the coefficient $c_{IJ}=0$ in this case ($|I|=1$). 

By the same reasoning all monomials with index set $I$ having only 2 elements come from the corresponding restriction  and the expansion in terms of the $e$'s, so $e_3(\t)|_I=0$ and $s_{(2,2)}(\t)|_{I}=(e_2(\t)e_2(\t))|_{I}$. The monomials whose first index has 2 elements are thus the following 
$$\sum_{\substack{ i_1\neq i_2,\\  j_1\leq j_2\; ; \; j_3 \leq j_4}} x_{i_1j_1}x_{i_2j_2}x_{i_1j_3}x_{i_2j_4} +\sum_{\substack{ i_1\neq i_2,\\ j_1\leq  j_2\; ; \; j_3 < j_4}} x_{i_1j_1}x_{i_2j_2}x_{i_2j_3}x_{i_1j_4}.$$
So we must have that the multiplicity of each index in $I$ is 2  and if $x_{IJ} \sim x_{i_1j_1}x_{i_2j_2}x_{i_1j_3}x_{i_2j_4}$ under commutation for any sequence $j_1,\ldots,j_4$, then $c_{IJ}=1$.  The alternative case is exactly when  $x_{IJ} \sim x_{i_1j_1}x_{i_1j_2}x_{i_2j_2}x_{i_2j_3}$ and $j_1,j_2,j_3$ are not necessarily distinct, then $c_{IJ}=0$.

Consider now the monomials  which have at least 3 distinct indices in $I$. If there are only 2 distinct indices in $J$ then we get the mirror sum of the above expression with the condition that the set of first indices has at least 3 distinct elements (to avoid double counting with the case $|I|=2$).

Let  $|I|\geq 3$ and $|J|\geq 3$. 

If $|I|=4$ and $|J|=4$ then all variables in $x_{IJ}$ commute with each other. The total coefficient is then $c_{IJ}=2$: there are $\binom{4}{2}=6$ ways to obtain $x_{IJ}$ from $h_2h_2$ by choosing  which two variables $x_{ij}$ come from the first $h_2$ and there are 4 ways to obtain it from $h_3h_1$ by choosing which variable comes from $h_1$. 

If $|I|=3$ and $|J|=4$ then $x_{IJ}=x_{i_1j_1}x_{i_1j_2}x_{i_2j_3}x_{i_3j_4}$ and $x_{i_1j_1}$ and $x_{i_1j_2}$ do not commute with each other, but all other pairs commute. The coefficient in $h_2(\t)h_2(\t)$ is 4 since $x_{i_1j_1}x_{i_1j_2}$ can come from the first $h_2(\t)$ fully, the second $h_2(\t)$ fully or both partially (i.e., $x_{i_1j_1}$ comes from the first $h_2(\t)$ and $x_{i_1j_2}$  from the second $h_2(\t)$). The corresponding coefficient in $h_3(\t)h_1(\t)$ is 3 since only $x_{i_1j_1}$ cannot come from $h_1(\t)$, so  we get $c_{IJ}=1$. 

If $|I|=3$ and $|J|=3$ the considerations depend on how the indices are distributed with respect to each other and a more careful analysis is needed.
Suppose $i_l = i_r$ and $j_l=j_r$. Then the remaining 2 variables commute with $x_{i_r,j_r}=x_{i_lj_l}$, so $x_{IJ} = x_{i_lj_l}x_{i_rj_r}... =0$.  

Let the repeating indices be $i\in I$ and $j\in J$, not both in the same variables. If $x_{ij}$ is not in $x_{IJ}$, then the variables $x_{i*}$ and $x_{*j}$ commute with each other. Let $x_{IJ} =x_{ia}x_{ib}x_{cj}x_{dj}$, then $[x_{IJ}]h_2(\t)h_2(\t) = 1$ since $x_{ia}x_{cj}$ must come from the first $h_2(\t)$ and $[x_{IJ}]h_3(\t)h_1(\t)=1$ since $x_{dj}$ must come from $h_1$, so $[x_{IJ}]s_{(2,2)}(\t)=c_{IJ}=0$. 

Suppose now that $x_{ij}$ appears in $x_{IJ}$ exactly once. There are four distinct commutation classes: $\sim x_{ia}x_{bj}x_{ij}x_{cd}$, $\sim x_{ia}x_{ij}x_{bj}x_{cd}$, $\sim x_{ij}x_{ia}x_{bj}x_{cd}$ $\sim x_{bj}x_{ij}x_{ia}x_{cd}$. For each such class we have the following coefficients in $h_2(\t)h_2(\t)$, $h_3(\t)h_1(\t)$ and $s_{(2,2)}(\t)$, derived by reasoning similar to the already used in the previous cases:

\medskip

\begin{tabular}{|c|c|c|c|c|}
\hline
$x_{IJ} \sim$ & $x_{ia}x_{bj}x_{ij}x_{cd}$ & $x_{ia}x_{ij}x_{bj}x_{cd}$ & $x_{bj}x_{ij}x_{ia}x_{cd}$ & $ x_{ij}x_{ia}x_{bj}x_{cd}$ \\
\hline
$[x_{IJ}] h_2(\t)h_2(\t)$ & 2  &1&  1 & 2 \\
$[x_{IJ}]h_3(\t)h_1(\t)$ & 1 & 1& 0 & 1 \\
$[x_{IJ}]s_{(2,2)}(\t)$ & 1 & 0 &1 & 1 \\
\hline
\end{tabular}  

\medskip

Last, if $|I|=4$ and $|J|=3$, then $[x_{IJ}]h_2(\t)h_2(\t)=2$ and $[x_{IJ}]h_3(\t)h_1(\t)=1$, so $c_{IJ}=1$.  


Noticing that we can write $c_{IJ}=0$ or $1$ whenever $x_{IJ}=0$ we can unify some of the cases and obtain the desired statement.
\end{proof}


\section{Final remarks}

The next step of the approach presented in this paper would be to derive an explicit nonnegative expansion for $s_{\la}(\t_1,\ldots,\t_k)$ when $\la=(n-k,k)$ is a two-row partition. The natural approach is to represent this Schur function via the  Jacobi-Trudi identity as $h_kh_{n-k}-h_{k-1}h_{n-k+1}$ and apply the known expansions for the homogeneous symmetric functions $h$. The main difficulty in this case is the apparent lack of a proper analogue of Lemma \ref{lemma:classes} which would enable the identification of monomials appearing in $h_kh_{n-k}$ and $h_{k-1}h_{n+1-k}$. However,  with the right interpretation and clever use of facts like Lemma \ref{lemma:rectangle_fit}, the current approach might be extendedable first to two-row partitions and then via a generalization to all shapes. 

\section*{Acknowledgements} The first author would like to thank Ricky I. Liu for interesting discussions regarding the coefficients 
$c_{\lambda}^D$. The second author would like to thank Nantel Bergeron for interesting discussions on the subject and the suggestion to do the $(2,2)$ case explicitly. The authors would also like to thank the anonymous referee for her careful reading, valuable suggestions and enthusiasm.


\end{document}